\documentclass[a4paper,10pt]{article}
\usepackage[utf8x]{inputenc}
\usepackage{tracefnt,amsmath,tabu,array}
\usepackage{amssymb,graphicx,setspace,amsfonts,amsbsy}
\usepackage{pifont,latexsym,ifthen,amsthm,rotating,calc,textcase,booktabs}

\newtheorem{theorem}{Theorem}[section]
\newtheorem{lemma}[theorem]{Lemma}

\newtheorem{proposition}[theorem]{Proposition}

\newtheorem{conjecture}[theorem]{Conjecture}
\newtheorem{remark}[theorem]{Remark}
\newcommand{\filledbox}{\leavevmode
  \hbox to.77778em{%
  \hfil\vbox to.675em{\hrule width.6em height.6em}\hfil}}

\newcommand{\Rm}{{\mathbb R}}

\newcommand{\eps}{\varepsilon}

\begin{document}
%\doublespacing
\tabulinesep=1.0mm
% Enter full title and short title for running headers
\title{Morawetz Estimates Method for Scattering of Radial Energy Sub-critical Wave Equation\footnote{MSC classes: 35L71, 35L05; The author is supported by National Natural Science Foundation of China Programs 11601374, 11771325}}

\author{Ruipeng Shen\\
Centre for Applied Mathematics\\
Tianjin University\\
Tianjin, China}

\maketitle

\begin{abstract}
  In this short paper we consider a semi-linear, energy sub-critical, defocusing wave equation $\partial_t^2 u - \Delta u = - |u|^{p -1} u$ in the 3-dimensional space with $p \in (3,5)$. We prove that if the energy of radial initial data $(u_0, u_1)$ outside a ball of radius $r$ centred at the origin decays faster than a certain rate $r^{-\kappa(p)}$, then the corresponding solution $u$ must scatter in both two time directions. The main tool of our proof is a more detailed version of the classic Morawetz estimate.
\end{abstract}

\section{Introduction}

We consider a defocusing semi-linear wave equation 
\[
 \left\{\begin{array}{ll} \partial_t^2 u - \Delta u = - |u|^{p-1}u, & (x,t) \in \Rm^3 \times \Rm; \\
 u(\cdot, 0) = u_0; & \\
 u_t (\cdot,0) = u_1. & \end{array}\right.\quad (CP1)
\]
This Cauchy problem is locally well-posed for any initial data $(u_0,u_1)$ in the critical Sobolev space $\dot{H}^{s_p} \times \dot{H}^{s_p-1}(\Rm^3)$ with $s_p \doteq 3/2 - 2/(p-1)$, as shown in Lindblad and Sogge's work \cite{ls}. There is also an energy conservation law for suitable initial data:
\[
 E(u, u_t) = \int_{\Rm^3} \left(\frac{1}{2}|\nabla u(\cdot, t)|^2 +\frac{1}{2}|u_t(\cdot, t)|^2 + \frac{1}{p+1}|u(\cdot,t)|^{p+1}\right)\,dx = \hbox{Const}.
\]
We then need to consider the global existence and asymptotic behaviour of solutions. The only fully understood case is the energy critical one with $p=5$. More than twenty years ago, M. Grillakis \cite{mg1} proved that any solution with initial data in the energy space $\dot{H}^1 \times L^2(\Rm^3)$ must scatter in both two time directions, i.e. the solution looks like a free wave as $t$ goes to infinity. A similar result is expected to hold for other $p$ as well.
\begin{conjecture}
 Any solution to (CP1) with initial data $(u_0,u_1) \in \dot{H}^{s_p} \times \dot{H}^{s_p-1}$ must exist for all time $t \in \Rm$ and scatter in both two time directions
\end{conjecture}
This is still an open problem in the field of analysis of PDEs, in spite of some progress. Roughly speaking, known results fall into two categories:
\paragraph{A priori estimate} The first type of results assume that a solution $u$ satisfies an a priori estimate
 \begin{equation}
  \sup_{t \in I} \left\|(u(\cdot,t), u_t(\cdot, t))\right\|_{\dot{H}^{s_p} \times \dot{H}^{s_p-1} (\Rm^3)} < +\infty \label{uniform UB}
 \end{equation}
in the whole lifespan $I$, then prove that $u$ is a global solution in time and scatters. Please see table \ref{type one results} for a list of these results. They are usually proved via a compactness-rigidity argument. Please note that our assumption \eqref{uniform UB} is automatically true in the energy critical case $p=5$, thanks to the energy conservation law. 
\begin{table}[h]
\caption{Results of scattering with a priori estimates in critical space}
\begin{center}
\begin{tabular}{|c|c|c|c|}\hline
 Dodson-Lawrie \cite{cubic3dwave} & Shen \cite{shen2} & Kenig-Merle \cite{km} & Killip-Visan \cite{kv2}\\
 \hline
 $1+\sqrt{2}<p\leq 3$, radial & $3<p<5$, radial & $p>5$, radial & $p>5$, non-radial\\
 \hline
\end{tabular}
\end{center}
\label{type one results}
\end{table}
\paragraph{Stronger assumptions on initial data} The second type of results make additional assumptions on the initial data in order to prove the scattering of solutions.
\begin{itemize}
 \item Conformal conservation laws (see \cite{conformal2, conformal}) can be used to prove the scattering of solutions for $p \in [3,5)$ if initial data satisfy an additional regularity-decay condition
  \begin{equation} \label{condition1}
    \int_{\Rm^3} \left[(|x|^2+1) (|\nabla u_0 (x)|^2 + |u_1(x)|^2) + |u_0(x)|^2 \right] dx < \infty.
  \end{equation}
The key ingredient of the proof is the following conformal conservation law
\[
  \frac{d}{dt} Q(t, u, u_t) = \frac{4(3-p)t}{p+1} \int_{\Rm^3} \left|u(x,t)\right|^{p+1} dx.
\]
Here $Q(t,\varphi,\psi) = Q_0(t, \varphi, \psi) + Q_1(t, \varphi)$ is called the conformal charge with
\begin{align*}
 Q_0(t,\varphi,\psi) &= \left\|x\psi + t \nabla \varphi \right\|_{L^2(\Rm^3)}^2 + \left\|(t\psi+2\varphi)\frac{x}{|x|} +|x|\nabla \varphi\right\|_{L^2(\Rm^3)}^2\\
 Q_1(t, \varphi) &= \frac{2}{p+1}\int_{\Rm^3} (|x|^2+t^2)|\varphi(x,t)|^{p+1} dx.
\end{align*}
The assumption \eqref{condition1} is essential to guarantee the finiteness of the conformal charge $Q(t,u,u_t)$ as defined above. The conformal conservation law then gives a global space-time integral 
\[
   \int_{|t|>1} \int_{\Rm^3} |u(x,t)|^{p+1}\,dxdt\lesssim_p \sup_{t\in\Rm} Q_1(t, u)  \leq \sup_{t\in\Rm} Q(t, u,u_t) = Q(0,u_0,u_1)< +\infty,
\]
which implies the scattering. One advantage of this argument is that the radial assumption is not necessary.  
\item In the author's previous work \cite{shen3} we proved the scattering of solutions if the radial initial data $(u_0,u_1) \in \dot{H}^1 \times L^2$ satisfy
\[
 \int_{\Rm^3} (1+|x|)^{1+2\eps} \left(|\nabla u_0|^2 +|u_1|^2\right) dx < \infty
\]
for a constant $\eps>0$. The assumptions on the decay of initial data are weaker than the conformal conservation law method above, reducing the exponent of $|x|$ from $2$ to slightly greater than $1$. The proof depends on a conformal transformation
\[
   v(y, \tau)  = \frac{\sinh |y|}{|y|} e^\tau u \left( e^\tau \frac{\sinh |y|}{|y|}\cdot y, t_0 + e^\tau \cosh |y|\right), \quad (y,\tau) \in \Rm^3 \times \Rm,
 \]
which converts a solution $u$ as above to a finite-energy solution $v$ of another non-linear wave equation
\[
 v_{\tau \tau} - \Delta_y v = - \left(\frac{|y|}{\sinh |y|}\right)^{p-1} e^{-(p-3)\tau} |v|^{p-1}v.
\]
This second equation turns out to be easier to deal with since its non-linear term has a good decay rate as $x$ or $t$ goes to infinity.
\end{itemize}

\paragraph{Main Result}
In this paper we prove the scattering result with even weaker assumptions on the decay rate of the initial data. 
\begin{theorem} \label{main}
Let $\kappa>\kappa(p) = \frac{3(5-p)}{p+3}$ be a constant. If initial data $(u_0,u_1)$ are radial and satisfy
\[
 \int_{\Rm^3} (|x|+1)^\kappa \left(\frac{1}{2}|\nabla u_0|^2 + \frac{1}{2}|u_1|^2 + \frac{1}{p+1}|u_0|^{p+1} \right) < +\infty.
\]
Then the corresponding solution $u$ to (CP1) must scatter in both two time directions. More precisely, there exists $(v_0^\pm ,v_1^\pm) \in (\dot{H}^1 \cap \dot{H}^{s_p}(\Rm^3))\times (L^2 \cap \dot{H}^{s_p-1}(\Rm^3))$, so that for any $s' \in [s_p,1]$
 \[
  \lim_{t \rightarrow \pm \infty} \left\|\begin{pmatrix} u(\cdot,t)\\  u_t(\cdot,t)\end{pmatrix} - 
  \mathbf{S}_L (t)\begin{pmatrix}u_0^\pm \\ u_1^\pm\end{pmatrix}\right\|_{\dot{H}^{s'} \times \dot{H}^{s'-1}(\Rm^3)} = 0.
 \]
 Here $\mathbf{S}_L (t)$ is the linear wave propagation operator.
\end{theorem}
\begin{remark}
 Given any initial data as in the theorem above, we have
\begin{align*}
 & \int_{\Rm^3} \left(|\nabla u_0|^q + |u_1|^q \right) dx \\
 & \qquad \leq 2\left[\int_{\Rm^3} \left(|\nabla u_0|^2 + |u_1|^2 \right)(1+|x|)^{\kappa}\, dx\right]^{q/2} \left[\int_{\Rm^3} \! (1+|x|)^{-\kappa q/(2-q)}\, dx\right]^{(2-q)/2}\\
 & \qquad < +\infty,
\end{align*}
as long as $\frac{6}{3+\kappa}<q<2$. By the Sobolev embedding $\dot{W}^{1,q} \times L^{q} \hookrightarrow \dot{H}^{s} \times \dot{H}^{s-1}$ with $\frac{1-s}{3} = \frac{1}{q}-\frac{1}{2}>0$, we have 
\[
 (u_0,u_1) \in \dot{H}^s \times \dot{H}^{s-1}(\Rm^3), \quad\hbox{for any} \; s\in\left(1-\frac{\kappa}{2},1\right].
\]
Since we have $s_p >\frac{5p-9}{2(p+3)}=1-\frac{\kappa(p)}{2}>1-\frac{\kappa}{2}$, our initial data is always contained in the critical Sobolev space.
\end{remark}
\begin{remark}
The author believes that the lower bound of decay rate $\kappa(p) = \frac{3(5-p)}{p+3}$ given in the main theorem is by no means optimal. However, this decay rate is still lower than previously known results.
\end{remark}
\paragraph{Notations} In this work we use the following notations.
\begin{itemize}
 \item If $u(x)$ is a radial function defined in $\Rm^3$, then by convention we define $u(r) = u(x)$ where $|x|=r$. 
 \item The notation $A \lesssim B$ means that there exists a constant $c$ so that the inequality $A \leq cB$ holds. We can also add one or more parameter(s) as the subscript of $\lesssim$. This implies that the constant $c$ depends on the parameter(s) mentioned but nothing else.
\end{itemize}
\section{Motivation} 

Because the initial data come with a finite energy, Energy-subcriticality leads to the global existence of the corresponding solution $u$. In order to obtain the scattering result, we need to use the following result:
\begin{proposition}[Scattering with a finite $L^{2(p-1)} L^{2(p-1)}$ norm, see Proposition 3.8 of \cite{shen3}] \label{L2p2}
 Let $u$ be a solution to (CP1) with initial data $(u_0,u_1) \in (\dot{H}^1 \cap \dot{H}^{s_p}) \times (L^2 \cap \dot{H}^{s_p-1})$. If $\|u\|_{L^{2(p-1)} L^{2(p-1)}(\Rm \times \Rm^3)} < \infty$, then $u$ scatters in both two time directions. More precisely, there exist two pairs $(u_0^\pm, u_1^\pm) \in (\dot{H}^1 \cap \dot{H}^{s_p}) \times (L^2 \cap \dot{H}^{s_p-1})$, so that the following limit holds for each $s' \in [s_p,1]$
\[
 \lim_{t \rightarrow \pm \infty} \left\|(u(\cdot,t), u_t(\cdot, t)) - \mathbf{S}_L(t) (u_0^\pm, u_1^\pm) \right\|_{\dot{H}^{s'} \times \dot{H}^{s'-1} (\Rm^3)} = 0. 
\]
\end{proposition}
\noindent As a result, it suffices to prove the global space-time integral estimate
\begin{equation} \label{desired space time integral}
 \int_{-\infty}^{\infty} \int_{\Rm^3} |u(x,t)|^{2(p-1)} dx dt < +\infty.
\end{equation}
The first known global space-time integral that comes into our mind is the Morawetz estimate
\[
 \int_{-\infty}^{\infty} \int_{\Rm^3} \frac{|u(x,t)|^{p+1}}{|x|} dx dt \lesssim E.
\]
In the energy critical case, i.e. $p=5$, we can apply inequality $|x|^{1/2}|u(x,t)| \lesssim \|u(\cdot,t)\|_{\dot{H}^1} \lesssim E^{1/2}$ for radial $\dot{H}^1$ functions and the Morawetz estimate immediately gives us \eqref{desired space time integral}. In the energy sub-critical case, however, if we applied the best estimate for radial solutions the author knows (See Lemma \ref{best decay} below)
\[ 
 |u(x,t)|\lesssim_p E^{\frac{2}{p+3}} |x|^{-\frac{4}{p+3}}
\]
we would obtain 
\[
 \int_{-\infty}^{\infty} \int_{\Rm^3} \frac{|u(x,t)|^{2(p-1)}}{|x|^{\frac{3(5-p)}{p+3}}} dx dt  
  = \int_{-\infty}^{\infty} \int_{\Rm^3} \frac{|u(x,t)|^{p+1}}{|x|}\cdot \left(|x|^\frac{4}{p+3}|u(x,t)|\right)^{p-3} dx dt < \infty.
\]
This is still weaker than the desired inequality \eqref{desired space time integral} as $|x|$ is large. In this work we will solve this problem by a suitable power-like decay 
\[
 \int_{-\infty}^{\infty} \int_{|x|>R} \frac{|u(x,t)|^{p+1}}{|x|} dx dt \lesssim R^{-\kappa}.
\]

\section{Review of Morawetz Estimates}

We are able to take a more careful look at this well-known global space-time integral estimate if we recall the original theorem given in Perthame and Vega's work \cite{benoit}.
\begin{theorem}
Let $u$ be a solution to (CP1) defined in a time interval $[0,T]$ with a finite energy $E$. Then given any $R>0$, we have the inequality
\begin{align}
 & \frac{1}{2R}\int_0^T \!\!\int_{|x|<R}(|\nabla u|^2+|u_t|^2) dx dt + \frac{1}{2R^2} \int_0^T \!\!\int_{|x|=R} |u|^2 d\sigma_R dt + \frac{p-2}{(p+1)R} \int_0^T \!\!\int_{|x|<R} |u|^{p+1} dx dt \nonumber \\
 & \qquad + \frac{p-1}{p+1} \int_0^T \!\!\int_{|x|>R} \frac{|u|^{p+1}}{|x|} dx dt + \frac{1}{R^2} \int_{|x|<R} |u(x,T)|^2 dx \leq 2E. \label{morawetz}
\end{align}
\begin{remark}
 We focus on the 3D case with $d=3$. Please note that the notations $E$ and $p$ were defined in a slightly different way in Perthame-Vega's original paper. Here we rewrite the inequality in the setting of the current work. The author also believes that there is a minor typing mistake in the original inequality. The last term $\frac{d^2-1}{4R^2} \int_{B(0,R)} |u(T)|^2 dx$ in the left hand side should have been $\frac{d^2-1}{8R^2} \int_{B(0,R)} |u(T)|^2 dx$ instead, although the change of this coefficient plays no role in the argument of this work.
\end{remark}
\end{theorem}
\paragraph{Careful look at Morawetz Estimate} First of all, let us ignore the final term in the left hand and substitute $T$ by $+\infty$. Thanks to the energy conversation law, we are also able to substitute the lower limit of the integrals by $-\infty$. Finally we can combine part of the third term above with the first term, then divide both sides by $2$ and write
\begin{align}
 & \frac{1}{2R}\int_{-\infty}^{+\infty} \int_{|x|<R}\left(\frac{1}{2}|\nabla u|^2+\frac{1}{2}|u_t|^2+\frac{1}{p+1}|u|^{p+1}\right) dx dt + \frac{1}{4R^2} \int_{-\infty}^{+\infty} \int_{|x|=R} |u|^2 d\sigma_R dt\nonumber\\
 & \qquad + \frac{p-3}{2(p+1)R} \int_{-\infty}^{+\infty} \int_{|x|<R} |u|^{p+1} dx dt +\frac{p-1}{2(p+1)} \int_{-\infty}^{+\infty} \int_{|x|>R} \frac{|u|^{p+1}}{|x|} dx dt  \leq E. \label{new morawetz}
\end{align}
Now we have an important observation that the first term in \eqref{new morawetz} is almost $E$ when $R$ is sufficiently large. In fact, the finite speed of propagation implies that for almost all $t \in (-R,R)$, as long as $|t|$ is not too close to $R$, almost all energy concentrates in the region $B(0,R) \doteq \{x\in \Rm^3: |x|<R\}$. This means that the values of other terms have to be very small. More precisely, we can calculate
\begin{align}
\frac{p-1}{2(p+1)} \int_{-\infty}^{+\infty} \int_{|x|>R} \frac{|u|^{p+1}}{|x|} dx dt 
 & \leq E - \frac{1}{2R}\int_{-R}^{+R} \int_{|x|<R}\left(\frac{1}{2}|\nabla u|^2+\frac{1}{2}|u_t|^2+\frac{1}{p+1}|u|^{p+1}\right) dx dt \nonumber\\
 & = \frac{1}{2R}\int_{-R}^{+R} \int_{|x|>R}\left(\frac{1}{2}|\nabla u|^2+\frac{1}{2}|u_t|^2+\frac{1}{p+1}|u|^{p+1}\right) dx dt. \label{key estimate}
\end{align}
The right hand side is exactly the average amount of energy which escapes outside the ball $B(0,R)$ for $t \in [-R,+R]$. Now we calculate carefully the energy outside the ball under additional decay assumption of the initial data. 

\section{An Energy Escaping Estimate}
Our argument relies on
\begin{proposition}
Let $u$ be a solution to (CP1) with a finite energy and satisfy
\[
 I = \int_{\Rm^3} |x|^{\kappa} \left(\frac{1}{2}|\nabla u_0|^2 + \frac{1}{2}|u_1|^2 + \frac{1}{p+1}|u_0|^{p+1}\right) dx < \infty.
\]
Then we have the function
\[
 I(t) = \int_{|x|>t} (|x|-t)^{\kappa} \left(\frac{1}{2}|\nabla u|^2 + \frac{1}{2}|u_t|^2 + \frac{1}{p+1}|u|^{p+1}\right) dx  \leq I, \quad t>0.
\]
\end{proposition}
\begin{proof}
 It immediately follows a basic calculation of the derivative
 \begin{align*}
  I'(t) = & -\int_{|x|>t} \kappa (|x|-t)^{\kappa-1} \left(\frac{1}{2}|\nabla u|^2 + \frac{1}{2}|u_t|^2 + \frac{1}{p+1}|u|^{p+1}\right) dx \\
  & \qquad + \int_{|x|>t} (|x|-t)^{\kappa} \left(\nabla u \cdot \nabla u_t + u_t u_{tt} + |u|^{p-1}uu_t\right) dx \\
  = & -\int_{|x|>t} \kappa (|x|-t)^{\kappa-1} \left(\frac{1}{2}|\nabla u|^2 + \frac{1}{2}|u_t|^2 + \frac{1}{p+1}|u|^{p+1}\right) dx \\
  & \qquad + \int_{|x|>t} \left\{(|x|-t)^{\kappa} \left(u_t u_{tt} + |u|^{p-1}uu_t\right)-u_t \hbox{\bf div} [(|x|-t)^{\kappa}\nabla u]\right\} dx \\
  = & -\int_{|x|>t} \kappa (|x|-t)^{\kappa-1} \left(\frac{1}{2}|\nabla u|^2 + \frac{1}{2}|u_t|^2 + \frac{1}{p+1}|u|^{p+1}\right) dx \\
  & \qquad + \int_{|x|>t} \left\{(|x|-t)^{\kappa} u_t\left(u_{tt} - \Delta u + |u|^{p-1}u\right)-\kappa u_t (|x|-t)^{\kappa-1} \frac{x}{|x|}\cdot \nabla u \right\} dx\\
  = & -\int_{|x|>t} \kappa (|x|-t)^{\kappa-1} \left(\frac{1}{2}|\nabla u|^2 + \frac{1}{2}|u_t|^2 + u_t \frac{x}{|x|}\cdot \nabla u+ \frac{1}{p+1}|u|^{p+1}\right) dx\leq 0.
 \end{align*}
 Here we have assumed that $u$ is sufficiently smooth. Otherwise we can apply smooth approximation techniques.
\end{proof}
\begin{remark} We can also consider the negative time direction and conclude
\[
 I(t) = \int_{|x|>|t|} (|x|-|t|)^{\kappa} \left(\frac{1}{2}|\nabla u|^2 + \frac{1}{2}|u_t|^2 + \frac{1}{p+1}|w|^{p+1}\right) dx  \leq I, \quad t \in \Rm.
\]
\end{remark}
\paragraph{Energy escaping the ball $B(0,R)$} Now we have ($t\in (-R,R)$)
\begin{align*}
 & \int_{|x|>R}\left(\frac{1}{2}|\nabla u|^2+\frac{1}{2}|u_t|^2+\frac{1}{p+1}|u|^{p+1}\right) dx\\
 & \qquad \leq (R-|t|)^{-\kappa} \int_{|x|>R} (|x|-|t|)^{\kappa} \left(\frac{1}{2}|\nabla u|^2 + \frac{1}{2}|u_t|^2 + \frac{1}{p+1}|w|^{p+1}\right) dx\\
 & \qquad \leq (R-|t|)^{-\kappa} I(t) \\
 & \qquad \leq (R-|t|)^{-\kappa} I.
\end{align*}
Combining this inequality with \eqref{key estimate}, we obtain the decay rate of space-time integral of $|u|^{p+1}/|x|$.
\begin{equation}\label{decay rate}
  \int_{-\infty}^{+\infty} \int_{|x|>R} \frac{|u|^{p+1}}{|x|} dx dt \lesssim_{p,\kappa} I R^{-\kappa}. 
\end{equation}

\section{Completion of the Proof}

Now we need the following point-wise estimate on solutions 
\begin{lemma} \label{best decay}
 If a radial function $u$ satisfies 
 \[
   \int_{\Rm^3} (|\nabla u|^2 + |u|^{p+1}) dx \leq E,
 \]
 then we have $|u(x)| \lesssim_{p} E^{2/(p+3)} |x|^{-4/(p+3)} $.
\end{lemma}
\begin{proof}
Let $|u(r_0)| = S$. For any $r\in(r_0,r_0+r_0^2 S^2/4E)$ we have
\begin{align*}
 |u(r)-u(r_0)| \leq &\int_{r_0}^r |u_r (s)| ds \leq \left(\int_{r_0}^r s^2 |u_r(s)|^2 ds\right)^{1/2}\left(\int_{r_0}^r s^{-2} ds\right)^{1/2} \\
 \leq & E^{1/2}\left(\frac{1}{r_0}-\frac{1}{r}\right)^{1/2} \leq \left[E\cdot\frac{r_0^2 S^2/4E}{r_0 r} \right]^{1/2} \leq \frac{S}{2}.
\end{align*}
Therefore $u$ satisfies $|u(r)| \geq S/2$ for all $r\in(r_0,r_0+r_0^2 S^2/4E)$. Now we use the $L^{p+1}$ bound 
\[
 \left(\frac{S}{2}\right)^{p+1} r_0^2 \cdot \frac{r_0^2 S^2}{4E} \leq \int_{r_0}^{r_0+\frac{r_0^2 S^2}{4E}} |u(r)|^{p+1} r^2 dr \lesssim \int_{\Rm^3} |u(x)|^{p+1} dx \leq E.
\]
This immediately gives the pointwise estimate.
\end{proof}
\paragraph{Global Integral Estimate} We start by applying Lemma \ref{best decay} and obtain
\begin{equation}\label{p minus 3}
 |u(x,t)|^{2(p-1)} = |u(x,t)|^{p-3} \cdot |u(x,t)|^{p+1} \lesssim_{p} E^{\frac{2(p-3)}{p+3}} |x|^{-\frac{4(p-3)}{p+3}} \cdot |u(x,t)|^{p+1}. 
\end{equation}
We use the inequality above, recall the decay rate estimate \eqref{decay rate} and deduce
\begin{align*}
 \int_{-\infty}^{\infty} \int_{|x|>R} \frac{|u(x,t)|^{2(p-1)}}{|x|^{\frac{3(5-p)}{p+3}}} dx dt & \lesssim_p \int_{-\infty}^{\infty} \int_{|x|>R} E^{\frac{2(p-3)}{p+3}} |x|^{-\frac{4(p-3)}{p+3}} \cdot \frac{|u(x,t)|^{p+1}}{|x|^{\frac{3(5-p)}{p+3}}} dx dt\\
& = E^{\frac{2(p-3)}{p+3}} \int_{-\infty}^{\infty} \int_{|x|>R} \frac{|u(x,t)|^{p+1}}{|x|} dx dt\\
& \lesssim_{p,\kappa}  E^{\frac{2(p-3)}{p+3}} I R^{-\kappa}.
\end{align*}
Since $\kappa>\frac{3(5-p)}{p+3}$, the inequality above implies  
\begin{equation} \label{outer estimate}
 \int_{-\infty}^{\infty} \int_{|x|>R} |u(x,t)|^{2(p-1)} dx dt \lesssim_{p,\kappa} E^{\frac{2(p-3)}{p+3}} I R^{-\left(\kappa-\frac{3(5-p)}{p+3}\right)}.
\end{equation}
This gives a finite upper bound for the integral of $|u|^{2(p-1)}$ in the region with large $x$. In order to find an upper bound of the integral in the region with small $x$, we can use \eqref{p minus 3} again and obtain
\[
 \int_{-\infty}^{\infty} \int_{|x|<R} \frac{|u(x,t)|^{2(p-1)}}{|x|^{\frac{3(5-p)}{p+3}}} dx dt \lesssim_p E^{\frac{2(p-3)}{p+3}} \int_{-\infty}^{\infty} \int_{|x|<R} \frac{|u(x,t)|^{p+1}}{|x|} dx dt \lesssim_p E^{\frac{3(p-1)}{p+3}}
\]
As a result we have
\begin{equation} \label{inner estimate}
 \int_{-\infty}^{\infty} \int_{|x|<R} |u(x,t)|^{2(p-1)} dx dt \lesssim_p E^{\frac{3(p-1)}{p+3}} R^{\frac{3(5-p)}{p+3}}.
\end{equation}
We choose an arbitrary $R>0$, combine \eqref{outer estimate} and \eqref{inner estimate} and finally conclude
\[
 \int_{-\infty}^{\infty} \int_{\Rm^3} |u(x,t)|^{2(p-1)} dx dt \leq C(p,\kappa,E,I) < +\infty.
\]
This finishes the proof.


\begin{thebibliography}{99}
% \bibitem{wavedr} J-P. Anker, P. Martinot, E. Pedon, and A. G. Setti. {``The shifted wave equation on Damek--Ricci spaces and on homogeneous trees''} \textit{Trends in Harmonic Analysis} (2013): 1-25.
 %\bibitem{wavehyper} J-P. Anker, V. Pierfelice, and M. Vallarino. {``The wave equation on hyperbolic spaces''} \textit{Journal of Differential Equations} 252(2012): 5613-5661.
 %\bibitem{fchain} M. Christ and M. Weinstein {``Dispersion of small amplitude solutions of the generalized Korteweg-de Vries equation''} \textit{Journal of Functional Analysis} 100(1991): 87-109.
% \bibitem{wkghyper} J-P. Anker, and V. Pierfelice. {``Wave and Klein-Gordon equations on hyperbolic spaces''} (2011), preprint arXiv: 1104.0177v2.
% \bibitem{ab1} A. Bulut. {``Global well-posedness and scattering for the defocusing energy-supercritical cubic nonlinear wave equation''} \textit{Journal of Functional Analysis} 263(2012): 1609-1660.
% \bibitem{ksph1} M. G. Cowling. {``The Kunze-Stein phenomenon''} \textit{Annals of Mathematics} 107(1978): 209-234.
%\bibitem{bahouri} H. Bahouri, and P. G\'{e}rard. {``High frequency approximation of solutions to critical nonlinear equations.''} \textit{American Journal of Mathematics} 121(1999): 131-175.
 %\bibitem{ctao} J. Colliander, M. Keel, G. Staffilani, H. Takaoka, and T. Tao. {``Global well-posedness and scattering in the energy space for the critical nonlinear nonlinear Schr\"{o}dinger equation in $\Rm^3$.''} \textit{Annals of Mathematics} 167(2007): 767-865.
 %\bibitem {truncation} M. Christ and A. Kiselev. {``Maximal functions associated to filtrations''} \textit{Journal of Functional Analysis} 179(2001): 409-425.
 %\bibitem{smallgs2} P. \'{D}Ancona, V. Georgiev, and H. Kubo. {``Weighted decay estimates for the wave equation''} \textit{Journal of Differential Equations} 177(2001): 146-208.
 \bibitem{cubic3dwave} B. Dodson and A. Lawrie. {``Scattering for the radial 3d cubic wave equation.''} \textit{Analysis and PDE}, 8(2015): 467-497.
 \bibitem{dkm2} T. Duyckaerts, C.E. Kenig, and F. Merle. {``Scattering for radial, bounded solutions of focusing supercritical wave equations.''} \textit{International Mathematics Research Notices} 2014:  224-258.
 %\bibitem{secret} T. Duyckaerts, C.E. Kenig, and F. Merle. {``Classification of radial solutions of the focusing, energy-critical wave equation.''} \textit{Cambridge Journal of Mathematics} 1(2013): 75-144.
 %\bibitem{tkm1} T. Duyckaerts, C.E. Kenig, and F. Merle. {``Universality of blow-up profile for small radial type II blow-up solutions of the energy-critical wave equation.''} \textit{The Journal of the European Mathematical Society} 13, Issue 3(2011): 533-599.
 %\bibitem{sob} M. Cowling, S. Giulini, S. Meda {``$L^p-L^q$ estimates for functions of the Laplace-Beltrami operator on noncompact symmetric spaces I''} \textit{Duke Mathematical Journal} 72(1993): 109-150.
% \bibitem{wavehyper97} J. Fontaine. {``A semilinear wave equation on hyperbolic spaces''} \textit{Communications in Partial Differential Equations} 22(1997): 633-659. 
 %\bibitem{superhyp} A. French. {``Existence and Scattering for Solutions to Semilinear Wave Equations on High Dimensional Hyperbolic Space''}(2014), Preprint Arxiv: 1407.2695v1.
 %\bibitem{smallgs3} V. Georgiev. \textit{Semilinear hyperbolic equations, MSJ Memoirs 7}, Tokyo: Mathematical Society of Japan, 2000.
% \bibitem{gwpwrn} V. Georgiev, H. Lindblad, and C. Sogge {``Weighted Strichartz estimates and global existence for semilinear wave equations''}, \textit{American Journal of Mathematics} 119(1997): 1291-1319.
 %\bibitem{locad1} J. Ginibre, A. Soffer and G. Velo. {``The global Cauchy problem for the critical nonlinear wave equation''} \textit{Journal of Functional Analysis} 110(1992): 96-130.
 \bibitem{conformal2} J. Ginibre, and G. Velo. {``Conformal invariance and time decay for nonlinear wave equations.''} \textit{Annales de l'institut Henri Poincar\'{e} (A) Physique th\'{e}orique} 47(1987), 221-276.
 %\bibitem{strichartz} J. Ginibre, and G. Velo. {``Generalized Strichartz inequality for the wave equation.''} \textit{Journal of Functional Analysis} 133(1995): 50-68.
 %\bibitem{blowup2} R. T. Glassey. {``Finite-time blow-up for solutions of nonlinear wave equations''} \textit{Mathematische Zeitschrift} 177(1981): 323-340.
 \bibitem{mg1} M. Grillakis. {``Regularity and asymptotic behaviour of the wave equation with critical nonlinearity.''} \textit{Annals of Mathematics} 132(1990): 485-509.
% \bibitem{mg2} M. Grillakis. {``Regularity for the wave equation with a critical nonlinearity.''} \textit{Communications on Pure and Applied Mathematics} 45(1992): 749-774.
 %\bibitem{fourier1} S. Helgason. {``Radon-Fourier transform on symmetric spaces and related group representations''}, \textit{Bulletin of the American Mathematical Society} 71(1965): 757-763.
 %\bibitem{fourier2} S. Helgason. \textit{Differential geometry, Lie groups, and symmetric spaces, Graduate Studies in Mathematics 34}, Providence: American Mathematical Society, 2001.
\bibitem{conformal} K. Hidano. {``Conformal conservation law, time decay and scattering for nonlinear wave equation''} \textit{Journal D'analysis Math\'{e}matique} 91(2003): 269-295.
 %\bibitem{SThyper} A. D. Ionescu. {``Fourier integral operators on noncompact symmetric spaces of real rank one''}, \textit{Journal of Functional Analysis} 174 (2000): 274-300.
% \bibitem{ksph2} A. D. Ionescu. {``An endpoint estimate for the Kunze-Stein phenomenon and related maximal operators''} \textit{Annals of Mathematics} 152, No.2(2000): 259-275.
 %\bibitem{hypersdg} A. D. Ionescu, and G. Staffilani. {``Semilinear Schr\"{o}dinger flows on hyperbolic spaces: scattering in $\Hm^1$''} \textit{Mathematische Annalen} 345, No.1(2012): 133-158.
 %\bibitem{blowup3} F. John. {``Blow-up of solutions of nonlinear wave equations in three space dimensions''} \textit{Manuscripta Mathematica} 28(1979): 235-268.
 %\bibitem{loc1} L. Kapitanski. {``Weak and yet weaker solutions of semilinear wave equations''} \textit{Communications in Partial Differential Equations} 19(1994): 1629-1676.
 %\bibitem{endpointStrichartz} M. Keel, and T. Tao. {``Endpoint Strichartz estimates''} \textit{American Journal of Mathematics} 120 (1998): 955-980.
 %\bibitem{channel} C. E. Kenig, A. Lawrie, B. Liu and W. Schlag. ``Channels of energy for the linear radial wave equation.'' \textit{Advances in Mathematics}  285(2015): 877-936.
 %\bibitem{kenig} C. E. Kenig, and F. Merle. {``Global Well-posedness, scattering and blow-up for the energy critical focusing non-linear wave equation.''} \textit{Acta Mathematica} 201(2008): 147-212.
\bibitem{kenig1} C. E. Kenig, and F. Merle. {``Global well-posedness, scattering and blow-up for the energy critical, focusing, non-linear Schr\"{o}dinger equation in the radial case.''} \textit{Inventiones Mathematicae} 166(2006): 645-675.
 \bibitem{km} C. E. Kenig, and F. Merle. {``Nondispersive radial solutions to energy supercritical non-linear wave equations, with applications.''} \textit{American Journal of Mathematics} 133, No 4(2011): 1029-1065.
 %\bibitem{kenig2} C. E. Kenig, and F.Merle. {``Scattering for $\dot{H}^{1/2}$ bounded solutions to the cubic, defocusing NLS in 3 dimensions.''} \textit{Transactions of the American Mathematical Society} 362(2010): 1937-1962.
 %\bibitem{kv1} R. Killip, B. Stovall and M. Visan. {``Blowup Behavior for the Nonlinear Klein-Gordan Equation.''}, preprint arXiv: 1203.4886v1.
 \bibitem{kv2} R. Killip, and M. Visan. {``The defocusing energy-supercritical nonlinear wave equation in three space dimensions''} \textit{Transactions of the American Mathematical Society}, 363(2011): 3893-3934.
 %\bibitem{kv3} R. Killip, and M. Visan. {``The radial defocusing energy-supercritical nonlinear wave equation in all space dimensions''} \textit{Proceedings of the American Mathematical Society}, 139(2011): 1805-1817.
 %\bibitem{kv} R. Killip, and M. Visan {``The focusing energy-critical nonlinear Schr\"{o}dinger equation in dimensions five and higher.''} \textit{American Journal of Mathematics} 132(2010): 361-424.
 %\bibitem{tao} R. Killip, T. Tao, and M. Visan. {``The cubic nonlinear Schr\"{o}dinger equation in two dimensions with redial data.''} \textit{Journal of the European Mathematical Society} 11, Issue 6(2009): 1203-1258.
% \bibitem{smallgs1} S. Klainerman, and G. Ponce. {``Global, small amplitude solutions to nonlinear evolution equations''} \textit{Communications on Pure and Applied Mathematics} 36(1983): 133-141.
 %\bibitem{negativeenergy} H. Levine. {``Instability and nonexistence of global solutions to nonlinear wave equations of the form $\mathbf{P} u_{tt} = −\mathbf{A} u + F(u)$''}, \textit{Transactions of the American Mathematical Society} 192(1974): 1-21.
 \bibitem{ls} H. Lindblad, and C. Sogge. {``On existence and scattering with minimal regularity for semi-linear wave equations''} \textit{Journal of Functional Analysis} 130(1995): 357-426.
 %\bibitem{Pecher} H. Pecher. {``Nonlinear small data scattering for the wave and Klein-Gordon equation.''} \textit{Mathematische Zeitschrift} 185(1984): 261-270.
 \bibitem{benoit} B. Perthame, and L. Vega. {``Morrey-Campanato estimates for Helmholtz equations.''} \textit{Journal of Functional Analysis} 164(1999): 340-355.
 %\bibitem{local1} H. Pecher. {``Nonlinear small data scattering for the wave and Klein-Gordon equation''} \textit{Mathematische Zeitschrift} 185(1984): 261-270.
 %\bibitem{shen1} R. Shen. {``Global well-posedness and scattering of defocusing energy subcritical nonlinear wave equation in dimension 3 with radial data.''} (2011): preprint arXiv: 1111.1234.
% \bibitem{ss1} J. Shatah, and M. Struwe. {``Regularity results for nonlinear wave equations''} \textit{Annals of Mathematics} 138(1993): 503-518.
 %\bibitem{ss2} J. Shatah, and M. Struwe. {``Well-posedness in the energy space for semilinear wave equations with critical growth''} \textit{International Mathematics Research Notices} 7(1994): 303-309.
 \bibitem{shen2} R. Shen. {``On the energy subcritical, nonlinear wave equation in $\Rm^3$ with radial data''}  \textit{Analysis and PDE} 6(2013): 1929-1987.
 \bibitem{shen3} R. Shen. {``Scattering of solutions to the defocusing energy subcritical semi-linear wave equation in 3D''} \textit{Communications in Partial Differential Equations} 42(2017): 495-518.
% \bibitem{subhyper} R. Shen and G. Staffilani. {``A Semi-linear Shifted Wave Equation on the Hyperbolic Spaces with Application on a Quintic Wave Equation on $\Rm^2$''}, \textit{Transactions of the American Mathematical Society} 368(2016): 2809-2864.
 %\bibitem{energyhyper} R. Shen. {``On the energy-critical semi-linear shifted wave equation on the hyperbolic spaces''}, to appear in \textit{Differential of Integral Equations}, preprint arXiv: 1408.0331.
% \bibitem{counter} T. Sideris. {``Nonexistence of global solutions to semilinear wave equations in high dimensions''} \textit{Journal of Differential Equations} 52(1984): 378-406. 
 %\bibitem{staten} R. J. Stanton and P. A. Tomas. {``Expansions for spherical functions on noncompact symmetric spaces''} \textit{Acta Mathematica} 140(1978): 251-271.
 %\bibitem{strcon} W. Strauss. \textit{Nonlinear wave equations, CBMS Regional Conference Series in Mathematics, Number 73}, Providence: American Mathematical Society, 1989.
 %\bibitem{struwe} M. Struwe. {``Globally regular solutions to the $u^5$ Klein-Gordon equation.''} \textit{Annali della Scuola Normale Superiore di Pisa - Classe di Scienze} 15(1988): 495-513.
 %\bibitem{pertao} T. Tao and M. Visan. {``Stability of energy-critical nonlinear Schr\"{o}dinger equations in high dimensions''} \textit{Electronic Journal of Differential Equations} 118(2005), 28.
% \bibitem{tataru} D. Tataru. {``Strichartz estimates in the hyperbolic space and global existence for the similinear wave equation''} \textit{Transactions of the American Mathematical Society} 353(2000): 795-807.
 %\bibitem{ktsutaya} K. Tsutaya. {``Scattering theory for semilinear wave equations with small data in two space dimensions''} \textit{Transactions of the American Mathematical Society} 342, No 2(1994): 595-618.
\end{thebibliography}
\end{document}